\documentclass[11pt]{amsart}

\usepackage{lmodern}
\usepackage{microtype}
\usepackage[english]{babel}
\usepackage{mathtools}
\usepackage{hyperref}
\usepackage[hyperpageref]{backref}
\usepackage[msc-links]{amsrefs}

\theoremstyle{plain} 
\newtheorem{theorem}{Theorem} 
\newtheorem{corollary}[theorem]{Corollary}

\DeclareMathOperator{\mre}{Re} 

\begin{document} 
\title{Extremals for linear functionals on $H^p$ are outer when $0<p<1$}

\begin{abstract}
	Let $L$ be a nontrivial bounded linear functional on the Hardy space $H^p$ of the unit disc. We prove that if $0<p<1$ and if $f$ in the unit sphere of $H^p$ satisfies $L(f)=\|L\|$, then $f$ is an outer function. The proof is short and relies on a construction due to de Leeuw and Rudin. Our result complements work of Duren, Romberg, and Shields on the failure of the Hahn--Banach theorem for $H^p$. When applied to the special case $L_k(f) = f^{(k)}(0)/k!$ for $k=1,2,3,\ldots$, it resolves two conjectures of the first-named author and Saksman and one conjecture of Suffridge.
\end{abstract}

\date{\today}


\thanks{Brevig is supported by Grant 354537 of the Research Council of Norway.}

\author{Ole Fredrik Brevig} 
\address{Department of Mathematical Sciences, Norwegian University of Science and Technology (NTNU), 7491 Trondheim, Norway} 
\email{ole.brevig@ntnu.no}

\author{\O istein S\o vik}
\address{IT Department, University of Oslo, 0373 Oslo, Norway}
\email{oistein.sovik@usit.uio.no}

\maketitle

\section{Introduction}
For $0<p\leq\infty$, we let $H^p$ denote the classical Hardy space of the unit disc $\mathbb{D}$, equipped with its usual (quasi-)norm $\|\cdot\|_p$. Suppose that $L$ is a nontrivial bounded linear functional on $H^p$. We say that a function $f$ in $H^p$ is an \emph{extremal} for $L$ if $\|f\|_p=1$ and $L(f)=\|L\|$.
 
\begin{theorem} \label{thm:outer}
    Suppose that $0<p<1$ and let $L$ be a nontrivial bounded linear functional on $H^p$. If $f$ is an extremal for $L$, then $f$ is outer.
\end{theorem}
 
Here we use the standard terminology that $f$ in $H^p$ is outer if the inner factor $I$ in its inner--outer factorization $f=IF$ is constant. We refer to Duren~\cite{Duren1970} for this and a few other rudimentary aspects of $H^p$ theory. 

The proof of Theorem~\ref{thm:outer} is short: it combines a construction due to de Leeuw and Rudin~\cite{dLR1958} with a suitable application of H\"older's inequality. The assumption that $0<p<1$ enters only via the H\"older exponents.
 
Since $H^p$ is a Banach space when $1\leq p\leq\infty$, the Hahn--Banach theorem implies that every $f$ in the unit sphere of $H^p$ is an extremal for some bounded linear functional. Duren, Romberg, and Shields~\cite{DRS1969}*{Theorem~13} showed that the Hahn--Banach theorem fails for $H^p$ when $0<p<1$ by exhibiting singular inner functions $S$ such that the only bounded linear functional on $H^p$ annihilating the proper closed subspace $SH^p$ is the trivial one. 

Theorem~\ref{thm:outer} demonstrates a different failure: if $f$ is a nontrivial function in $H^p$ that is not outer, then $\lvert L(f)\rvert<\|f\|_p$ for every bounded linear functional $L$ on $H^p$ with $\|L\|\leq1$. We can actually say slightly more.
 
\begin{corollary}\label{cor:hb}
Suppose that $0<p<1$.
\begin{enumerate}
	\item[(a)] If $f$ is a nontrivial function in $H^p$ that is not outer, then
	\[\sup_{\|L\|\leq1}\lvert L(f)\rvert<\|f\|_p.\]
	\item[(b)] If $L$ is a bounded linear functional on $H^p$, then
	\[\|L\|=\sup\{\lvert L(f)\rvert : \|f\|_p\leq1 \text{ and } f \text{ is outer}\}.\]
\end{enumerate}
\end{corollary}
 
Corollary~\ref{cor:hb}~(b) also holds for $H^1$, as a direct consequence of the result of de Leeuw and Rudin~\cite{dLR1958} that every function in the unit ball is a convex combination of two outer functions in the unit sphere.
 
Consider for $k=1,2,3,\ldots$ the bounded linear functional
\[L_k(f)=\frac{f^{(k)}(0)}{k!}.\]
A normal families argument shows that $L_k$ has extremals. For $0<p<1$, it follows from a special case of a result of S.~Ya.~Khavinson~\cite{Khavinson1986}*{Corollary~II.1.7} (see also~\cite{BS2020}*{Lemma~4}) that every extremal takes the form
\[f=BP^{2/p},\]
where $B$ is a finite Blaschke product and $P$ is an outer polynomial, both of degree at
most $k$. Combining this with Theorem~\ref{thm:outer}, we obtain the following.
 
\begin{corollary}\label{cor:coeff}
    Suppose that $0<p<1$ and $k\geq1$. Every extremal for $L_k$ on $H^p$ is of the form $f=P^{2/p}$, where $P$ is an outer polynomial of degree $k$.
\end{corollary}
 
Corollary~\ref{cor:coeff} confirms two conjectures made by the first-named author and Saksman~\cite{BS2020}*{Section~5.1}: no extremal for $L_k$ on $H^p$ with $0<p<1$ vanishes in $\mathbb{D}$ and $\|L_k\|<\|L_{k+1}\|$ for $k=1,2,3,\ldots$.
 
S.~Ya.~Khavinson~\cite{Khavinson1986}*{Theorem~II.2.8} and Suffridge~\cite{Suffridge1990}*{Theorem~1} studied the extremal problem for $L_k$ restricted to nonvanishing functions, and showed that its extremals take the form
\[f=SP^{2/p},\]
where $P$ is an outer polynomial of degree at most $k$ and $S$ is a singular inner function. Suffridge~\cite{Suffridge1990}*{p.~187} conjectured that $S$ is constant when $0<p<1$. In fact, if $0<p<1$, then the extremal problem for any bounded linear functional $L$ on $H^p$ is equivalent to its restriction to nonvanishing functions: they have the same value by Corollary~\ref{cor:hb}~(b), and the same extremals by Theorem~\ref{thm:outer}. In particular, the extremals of the restricted problem for $L_k$ are precisely those described in Corollary~\ref{cor:coeff}, which confirms Suffridge's conjecture.

\subsection*{AI disclosure}
An initial proof of Theorem~\ref{thm:outer} was found by Claude Opus~5.5 when it was tasked with resolving the conjectures from~\cite{BS2020}. That proof was variational and based on ideas going back to F.~Riesz~\cite{Riesz1920}. Corollary~\ref{cor:hb}~(b) led the authors to the work of de Leeuw and Rudin~\cite{dLR1958}, where they found the main idea in the much simpler proof presented below. The authors wrote the final version of the paper and take full responsibility for its content.

\section{Proofs}

\begin{proof}[Proof of Theorem~\ref{thm:outer}]
    Let $f= IF$ be a function in the unit sphere of $H^p$. Following de Leeuw and Rudin~\cite{dLR1958}*{\S2.2}, we write $f=(f_+-f_-)/2$, where
    \[f_\pm=\frac{(1\pm I)^2F}{2}.\]
    Since $I$ is inner, we have $\lvert f_\pm \rvert = (1\pm \mre{I}) \lvert F\rvert$ almost everywhere on the unit circle $\mathbb{T}$. Using next that $\|F\|_p = \|f\|_p=1$ and H\"older's inequality with exponents $1/p$ and $1/(1-p)$, it follows that
    \[\|f_\pm\|_p \leq \int_0^{2\pi} (1\pm \mre{I(e^{i\theta})}) \,\lvert F(e^{i\theta})\rvert^p \,\frac{d\theta}{2\pi} = 1 \pm \int_0^{2\pi} \mre{I(e^{i\theta})}\,\lvert F(e^{i\theta})\rvert^p \,\frac{d\theta}{2\pi}.\]
    Since $F$ is outer, equality holds in this application of H\"older's inequality if and only if $1\pm \mre{I}$ is constant almost everywhere on $\mathbb{T}$. A function in $H^\infty$ whose real part is constant almost everywhere on $\mathbb{T}$ is itself constant, so
    \begin{equation} \label{eq:strict}
        \frac{\|f_+\|_p + \|f_-\|_p}{2} < 1
    \end{equation}
    unless $I$ is constant. If $f$ is an extremal for $L$, then
    \[\|L\| = L(f) = \frac{L(f_+)-L(f_-)}{2} \leq \frac{\lvert L(f_+)\rvert +\lvert L(f_-)\rvert}{2} \leq \|L\| \frac{\|f_+\|_p+\|f_-\|_p}{2}.\]
    Since $\|L\|\neq0$, this shows that \eqref{eq:strict} does not hold, so that $I$ is constant.
\end{proof}

\begin{proof}[Proof of Corollary~\ref{cor:hb}]
    We begin with (a). The last two inequalities in the final display of the proof of Theorem~\ref{thm:outer} hold for every $f$ in $H^p$, and show that \[\lvert L(f)\rvert\leq \frac{\|f_+\|_p+\|f_-\|_p}{2}\] 
    for every bounded linear functional $L$ with $\|L\|\leq1$. Applying \eqref{eq:strict} to $f/\|f\|_p$ shows that the right-hand side is strictly less than $\|f\|_p$ unless $f$ is outer. 

    We continue with (b). Assume that $L$ is nontrivial. For $0<r<1$, let $L_r(f)=L(f_r)$, where $f_r(z)=f(rz)$. Then $L_r$ is nontrivial and $\|L_r\|\to\|L\|$ as $r\to1^-$, since $\|f_r\|_p\leq\|f\|_p$ and $f_r\to f$ in $H^p$ (see~\cite{Duren1970}*{Theorem~2.6}). By a normal families argument, $L_r$ has an extremal $F$, which is outer by Theorem~\ref{thm:outer}. Hence $F_r$ is outer, $\|F_r\|_p\leq1$, and $L(F_r)=L_r(F)=\|L_r\|\to\|L\|$.
\end{proof}

\begin{proof}[Proof of Corollary~\ref{cor:coeff}]
Let $f$ be an extremal for $L_k$. It follows at once from Theorem~\ref{thm:outer} and e.g.~\cite{BS2020}*{Lemma~4} that $f=P^{2/p}$ for an outer polynomial \[P(z)=a_0+a_1z+\cdots+a_kz^k.\] Applying~\cite{BS2020}*{Lemma~3} with $g=h=P$ and $\varphi(z)=z^k$, we find that
\[P(0)^{2/p}=L_k(\varphi f)=\|L_k\|\langle \varphi,\lvert P\rvert^2\rangle=\|L_k\|\,a_0\overline{a_k}.\]
The left-hand side is nonzero since $P$ is outer, so $a_k\neq0$.
\end{proof}

\bibliography{exouter}

\end{document}